\documentclass{amsart}
\usepackage[twoside,lmargin=9em,rmargin=9em,top=4cm,bottom=4cm]{geometry}
\usepackage{graphicx,xcolor}
\usepackage{color}

\usepackage{amsmath}
\usepackage{amssymb}
\usepackage{amsthm}
\usepackage{tikz}
\usepackage[all]{xy}
\usepackage{hyperref}

\newtheorem{thm}{Theorem}[section]
\newtheorem{prop}[thm]{Proposition}
\newtheorem{lemma}[thm]{Lemma}

\usepackage{amsthm}

\theoremstyle{remark}
\newtheorem{rem}[thm]{Remark}
\newtheorem{example}[thm]{Example}

\numberwithin{equation}{section}

\def\CC{\mathord{\mathbb{C}}}

\def\RR{\mathord{\mathbb{C}}}

\newcommand{\bm}[1]{\mbox{\boldmath{$#1$}}}

\usepackage[colorinlistoftodos]{todonotes}

\title[relativistic Toda lattice of type B]{relativistic Toda lattice of type B and quantum $K$-theory of type C flag variety}
\author[T.~Ikeda]{Takeshi Ikeda}
\address{Faculty of Science and Engineering, Waseda University, 3-4-1 Okubo, Shinjuku-ku, Tokyo 169-8555 Japan}
\email{gakuikeda@waseda.jp}

\author[S.~Iwao]{Shinsuke Iwao}
\address{Faculty of Business and Commerce, Keio University, 4–1–1 Hiyosi, Kohoku-ku, Yokohama-si, Kanagawa 223-8521, Japan.}
\email{iwao-s@keio.jp}

\author[T.~Kouno]{Takafumi Kouno}
\address{Faculty of Science and Engineering, Waseda University, 3-4-1 Okubo, Shinjuku-ku, Tokyo 169-8555 Japan}
\email{t.kouno@aoni.waseda.jp}

\author[S.~Naito]{Satoshi Naito} 
\address{Department of Mathematics, Institute of Science Tokyo, 2-12-1 Oh-Okayama, Meguro-ku, Tokyo 152-8551 Japan}
\email{naito@math.titech.ac.jp}

\author[K.~Yamaguchi]{Kohei Yamaguchi}
\address{Faculty of Science and Engineering, Waseda University, 3-4-1 Okubo, Shinjuku-ku, Tokyo 169-8555 Japan}
\email{yamaguchi\_86@aoni.waseda.jp}

\begin{document}
\maketitle

\begin{abstract}
We introduce a classical integrable system associated with the torus-equivariant quantum $K$-theory of type C flag variety. We prove that its conserved quantities coincide with the generators of the defining ideal of the Borel presentation of the quantum $K$-ring obtained by Kouno and Naito. In particular, the Hamiltonian of the system is naturally regarded
as a type B analogue of the relativistic Toda lattice introduced by Ruijsenaars. We also construct B\"acklund transformations describing the discrete time evolution of the system.

This construction makes explicit the integrable structure underlying the quantum $K$-theory and provides a framework for further studies of the $K$-theoretic Peterson isomorphism.
\end{abstract}

\section{Introduction}

Since Givental and Lee \cite{GL03} established a connection between the quantum $K$-theory of $G/B$
and the $q$-difference Toda lattice in type~A, 
and suggested an extension to other types,
the interaction between quantum $K$-theory and quantum integrable systems has attracted considerable attention; see for example, \cite{B04,FFJMM09,GLO10}.
In the simply laced case, the conjectural picture was settled
by \cite{BF14}.
For the non-simply laced case, a representation-theoretic extension was developed in \cite{BF17}, although its geometric realization has not yet been completely clarified.

Recently, when $G$ is a symplectic group, a Borel presentation of the quantum $K$-ring was obtained in \cite{KN}.
The purpose of this paper is to investigate the \emph{classical} integrable system underlying this Borel presentation.
In type~A, the relation between the quantum $K$-theory of $G/B$ and the relativistic Toda lattice has been studied; see \cite{IIM20,IIN24,IINY25,MNS25}.
The classical integrable system introduced here for $G/B$ in type C has the property that its conserved quantities coincide with the generators of the ideal giving the ring presentation in \cite{KN}.
Its Hamiltonian can be naturally regarded as a type B
analogue of the relativistic Toda lattice introduced by Ruijsenaars \cite{ruijsenaars1990relativistic}.

The relativistic Toda lattice associated with a simple root system has been studied from various aspects.
Ruijsenaars~\cite{ruijsenaars1990relativistic} derived the relativistic Toda lattices of types C and BC as a certain restriction of the Ruijsenaars-Schneider models~\cite{ruijsenaars1986} (also known as the relativistic Calogero-Moser systems).
In \cite{vandiejen1994}, van Diejen introduced the difference Cologero-Moser system (also known as the van Diejen model) whose classical limit induces the Ruijsenaars-Schneider model of type BC.
We also note that several Lax formulations of the relativistic Toda lattice have been studied in the literature.
Chen-How-Yang \cite{chen2000arxiv,chen2001} presented Lax pairs for the Ruijsenaars-Schneider models of types C and BC via reduction of that of type A, and showed their complete integrability. 
Pusztai~\cite{pusztai2020} established a Lax formalism for a deformed relativistic Toda lattice with parameters, which includes the relativistic Toda lattice of type C.
More recently, K.~Lee and N.~Lee~\cite{lee2024dimers,lee2025dimers} presented a quantum-mechanical method for constructing a $2\times 2$ Lax pair for the relativistic Toda lattices of types A, B, C, and D.

The classical integrable system introduced in this paper is expected to be related to the $\mathrm{B}_n$-type $q$-difference Toda system, which should admit both a construction from a suitable limit of the commuting family of $q$-difference operators introduced by van Diejen \cite{vandiejen1994} and a Whittaker-type construction in the sense of Gonin--Tsymbaliuk \cite{GoninTsymbaliuk}. This connection will be investigated in a separate work.

\subsection*{Organization}
In Section \ref{sec:Int_system}, we construct a \(2n\times 2n\) Lax matrix $L$ \eqref{eq:matrix_expression_of_L}, whose characteristic polynomial gives the defining ideal of the torus-equivariant quantum $K$-ring of the type C flag variety (Theorem \ref{thm:relation_to_Ktheory}).
Also, we show that the Lax equation \eqref{eq:Lax_equation} realizes the relativistic Toda lattice of type \(\mathrm{B}_n\) by determining its phase space, the flow of motion, and the Poisson structure.
Our approach is similar to the method employed by Kostant~\cite{kostant1979} for the (non-relativistic) Toda lattice; all computations are carried out within the classical framework.

In Section \ref{sec:Transformation}, we derive a B\"{a}cklund transformation by using the factorization of the Lax matrix.
This can be regarded as a time evolution of the discrete-time relativistic Toda lattice (see~\cite{suris1990,suris1996}).

\subsection*{Acknowledgments}
The authors are grateful to Kanehisa Takasaki and Yasuhiro Ohta for helpful discussions.
T.I. and S.I. were partially supported by the Grant-in-Aid for Scientific Research (C) 22K03239.
S.I. was partially supported by the Grant-in-Aid for Scientific Research (C) 23K03056.
T.K. was partially supported by the Grant-in-Aid for Research Activity Start-up 24K22842. S.N. was partly supported by the Grant-in-Aid for Scientific Research (C) 21K03198.

\section{The integrable system}\label{sec:Int_system}

\subsection{Preliminaries: a Lie group decomposition}
In this paper, we adopt a matrix-based approach to construct an integrable system.
Throughout, we employ the following matrix factorization instead of the commonly used Gauss decomposition.
Let $G=GL_{N}(\RR)$ be the general linear Lie group over the complex numbers.
In the following, we suppose $N=n$ or $N=2n$.
We denote by $B_+$ (resp.~$B_{-}$) the Borel subgroup of upper (resp.~lower)
triangular matrices in $G$, and by $U_{+}\subset B_+$ (resp.~$U_{-}\subset B_-$) the subgroup of upper (resp.~lower) triangular unipotent matrices.

Set $J=\sum_{i=1}^{n}E_{i,n+1-i}$, where $E_{ij}$ is the matrix unit.
We define two subgroups $G_+,G_-\subset G$ by
\[
\begin{gathered}
G_+=
\left.\left\{
\left[
\begin{array}{c|c}
X & Y \\\hline
O & Z 
\end{array}
\right]\in GL_{2n}(\RR)\;\right|\;
X\in B_+,\,
Z\in U_+
\right\},\\
G_-=\left.
\left\{
\left[
\begin{array}{c|c}
U & O \\\hline
V & W 
\end{array}
\right]\in GL_{2n}(\RR)\;\right|\;
W=JUJ
\right\}.
\end{gathered}
\]
It is straightforward to check that both $G_+$ and $G_-$ are of dimension $2n^2$.

\begin{prop}\label{prop:mat_decomp}
A generic $2n\times 2n$ matrix $X\in GL_{2n}(\RR)$ can be uniquely decomposed as $X=KR$, where $K\in G_-$ and $R\in G_+$.
\end{prop}
\begin{proof}
By the Gauss decomposition, a generic $X\in GL_{2n}(\RR)$ can be uniquely decomposed as $X=U_1R_1$, where $U_1\in B_-$ and $R_1\in U_{+}$.
Since $U_{+}$ is a subgroup of $G_+$, it follows that $R_1\in G_+$.
Therefore, it suffices to prove the proposition in the case where $X$ is lower triangular.

Assume $X=
\left[
\begin{array}{c|c}
A&O\\\hline
B&C\\
\end{array}
\right]$ to be lower triangular.
Then, we have $A, C\in B_-$. 
Consider the matrix $C^{-1}JAJ$ and its Gauss decomposition:
\[
C^{-1}JAJ=R_2U_2\qquad (R_2\in U_{+},\,U_2\in B_-).
\]
Then, by putting
\[
P=JU_2J,\quad Q=R_2^{-1},\quad U=AJU_2^{-1}J,\quad V=BJU_2^{-1}J,\quad W=CR_2,
\]
we obtain the matrix factorization $X=\left[
\begin{array}{c|c}
U&O\\\hline
V&W\\
\end{array}
\right]
\left[
\begin{array}{c|c}
P&O\\\hline
O&Q\\
\end{array}
\right]
$.
It is straightforward to show
\[
\left[
\begin{array}{c|c}
U&O\\\hline
V&W\\
\end{array}
\right]\in G_-\quad \text{and}\quad
\left[
\begin{array}{c|c}
P&O\\\hline
O&Q\\
\end{array}
\right]\in G_+,
\]
which conclude the proposition.
\end{proof}

\subsection{The Lax matrix via quantum $K$-theory}\label{sec:Lax_matrix}

Let $(z_1,\dots,z_n,Q_1,\dots,Q_n)$ be a set of complex variables. 
We introduce a $2n\times 2n$ \textit{Lax matrix} $L$ by
\begin{equation}\label{eq:matrix_expression_of_L}
L=
NBC^{-1},
\end{equation}
where
\[
N=
\left[
\begin{array}{c|c}
N_{11} & O\\\hline
O & N_{22}
\end{array}
\right],\quad
B=
\left[
\begin{array}{c|c}
\bm{1}_n & J\\\hline
O & \bm{1}_n
\end{array}
\right],\quad
C=
\left[
\begin{array}{c|c}
JN_{22}J & O\\\hline
P & JN_{11}J
\end{array}
\right]
\]
and
\[
N_{11}
=\begin{bmatrix}
z_1 & 1\\
&z_2 & \ddots\\
&& \ddots & 1\\
&& & z_n
\end{bmatrix},\qquad
N_{22}
=\begin{bmatrix}
1 & Q_{n-1}z_{n-1}\\
&1 & \ddots\\
&& \ddots & Q_1z_1\\
&& & 1
\end{bmatrix},\quad
P=Q_nz_n\cdot E_{1,n}.
\]

\begin{example}
When $n=2$, the Lax matrix $L$ is expressed as follows:
\[
L=
\begin{bmatrix}
(1-Q_1)z_1 & 1 & 0 & 1\\
-Q_1(1-Q_2)z_1z_2 & (1-Q_2)z_2 & 1 & 0\\
(1-Q_1)Q_1Q_2z_1 & -(1-Q_1)Q_2 & (1-Q_1)z_2^{-1} & Q_1\\
-Q_1Q_2 & Q_2z_1^{-1} & -z_1^{-1}z_2^{-1} & z_1^{-1}
\end{bmatrix}.
\]
\end{example}

\begin{rem}
In~\cite[Eq.~(3.117)]{pusztai2020}, Pusztai introduced a $2n\times 2n$ Lax matrix for a deformed relativistic Toda lattice with one parameter $\kappa$.
At $\kappa=0$ limit, it reduces to the relativistic Toda lattice of type C.
This matrix has a form similar to our $L$, but its relationship remains unclear.
\end{rem}

For $1\leq i\leq 2n$, let $(-1)^iF_i(z_1,\dots,z_n,Q_1,\dots,Q_n)$ be the coefficient of $\lambda^{2n-i}$ in the characteristic polynomial $\det(\lambda E-L)$:
\[
\det(\lambda E-L)=\sum_{i=0}^{2n}(-1)^{i}F_i(z_1,\dots,z_n,Q_1,\dots,Q_n)\lambda^{2n-i}.
\]

Standard arguments based on the Lindstr\"{o}m–Gessel–Viennot theorem (LGV theorem, see \cite[Theorem~5.4.1]{Stanley1999}) show that the polynomials $F_i$ admit a combinatorial expression as described in the following.
Consider the weighted graph shown in Figure \ref{fig:graph_description}.
\begin{figure}[htbp]
    \centering
\begin{tikzpicture}[scale = .9, yscale = 1.2]
\foreach \x in {2} {
\foreach \y in {2,3,5,6} {\draw[fill] (\x,\y) circle[radius = 2pt];}}%
\foreach \x in {5} {
\foreach \y in {1,2,4,5} {\draw[fill] (\x,\y) circle[radius = 2pt];}}%
\draw(0,6)node[left]{$L_1$}--node[pos=.1,above]{$z_1$}(12,6)node[right]{};
\draw(0,5)node[left]{$L_2$}--node[pos=.1,above]{$z_2$}(12,5)node[right]{};
\draw(0,4)node[left]{$L_3$}--node[pos=.1,above]{$z_3$}(12,4)node[right]{};
\draw(0,3)node[left]{$L_{\overline{3}}$}--node[pos=.1,above]{$z_3^{-1}$}(12,3)node[right]{};
\draw(0,2)node[left]{$L_{\overline{2}}$}--node[pos=.1,above]{$z_2^{-1}$}(12,2)node[right]{};
\draw(0,1)node[left]{$L_{\overline{1}}$}--node[pos=.1,above]{$z_1^{-1}$}(12,1)node[right]{};
\draw (2,6)--node[right]{$-Q_1$}(5,5);
\draw (2,5)--node[right]{$-Q_2$}(5,4);
\draw (2,3)--node[right]{$-Q_2$}(5,2);
\draw (2,2)--node[right]{$-Q_1$}(5,1);
\draw[dashed,thick] (8,4)node[circle,fill,inner sep=1.5pt] {} --node[left]{$-Q_3$}(8.5,3)node[circle,fill,inner sep=1.5pt] {};
\draw[dashed,thick] (8.5,5)node[circle,fill,inner sep=1.5pt] {}--node{$-Q_2Q_3$}(10,2)node[circle,fill,inner sep=1.5pt] {};
\draw[dashed,thick] (9,6)node[circle,fill,inner sep=1.5pt] {}--node[right]{$-Q_1Q_2Q_3$}(11.5,1)node[circle,fill,inner sep=1.5pt] {};
\end{tikzpicture}
\caption{An example of a weighted graph for $n=3$.
There are six horizontal lines $L_k$ and $L_{\overline{k}}$ ($k=1,2,3$),
four short segments, and three dashed segments.
 Weights are assigned to some segments as indicated alongside them.} 

\label{fig:graph_description}
\end{figure}
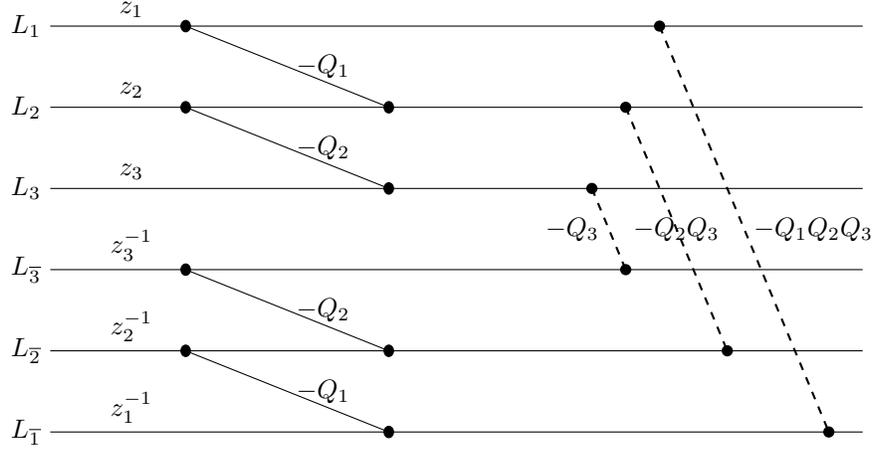
Let 
\[
I=\{1<2<\dots<n<\overline{n}<\dots<\overline{2}<\overline{1}\}
\]
be a totally ordered set consisting of $2n$ elements.
Consider $2n$ horizontal lines, arranged in this order
and indexed by the elements of $I$.
Let $L_x$ denote the line corresponding to $x\in I$.
There are $2(n-1)$ short segments connecting adjacent lines,
except between $L_n$ and $L_{\overline{n}}$,
and $n$ dashed segments connecting $L_k$ and $L_{\overline{k}}$.
The weight $-Q_k$ is assigned to 
the short segment connecting $L_k$ and $L_{k+1}$, and the weight $-Q_{k-1}$ is assigned to
the segment connecting $L_{\overline{k}}$ and $L_{\overline{k-1}}$.
The weight $-Q_kQ_{k+1}\cdots Q_n$ is assigned to 
the dashed segment connecting $L_{k}$ and $L_{\overline{k}}$.
In addition, $z_k$ and $z_k^{-1}$ are assigned to the leftmost segments on $L_k$ and $L_{\overline{k}}$, respectively.
Throughout, each (directed) path proceeds from left to right.
For $x\leq y\in I$, let $\gamma(x,y)$ denote the shortest path in the weighted graph connecting the left endpoint of $L_x$ and the right endpoint of $L_y$.

Let $w(x,y)$ be the \textit{weight} of the path $\gamma(x,y)$, the product of all weights along it.
Then, by an LGV-type argument, we have the following combinatorial description of $F_i$:
\begin{equation}\label{eq:path_description_of_F}
F_i=\sum_{
\substack{
x_1\leq y_1<x_2\leq y_2<\dots<x_i\leq y_i
}
}\prod_{l=1}^i
w(x_l,y_l).
\end{equation}
A proof of \eqref{eq:path_description_of_F} will be given in Appendix \ref{sec:appendix}.

Explicitly, we have
\[
w(x,y)=
\begin{cases}
z_k & ((x,y)=(k,k))\\
-Q_kz_k & ((x,y)=(k,k+1))\\
z_k^{-1} & ((x,y)=(\overline{k},\overline{k}))\\
-Q_{k-1}z_k^{-1} & ((x,y)=(\overline{k},\overline{k-1}))\\
-z_kQ_kQ_{k+1}\cdots Q_n & ((x,y)=(k,\overline{k})\\
z_kQ_kQ_{k+1}\cdots Q_n & ((x,y)=(k,\overline{k+1}))\\
0 & (\text{otherwise}).
\end{cases}
\]
\begin{example}
From \eqref{eq:path_description_of_F}, we have
\[
F_1=(1-Q_1)z_1+\dots+(1-Q_n)z_n+(1-Q_{n-1})z_n^{-1}+\dots+(1-Q_1)z_2^{-1}+z_1^{-1}.
\]
\end{example}

\begin{example}
Let $n=2$ and $i=2$.
The contribution $w(x_1,y_1,x_2,y_2):=w(x_1,y_1)w(x_2,y_2)$ is expressed as follows:
\begin{gather*}
w(1,1,2,2)=z_1z_2,\quad
w(1,1,2,\overline{2})=-z_1z_2Q_2,\quad
w(1,1,\overline{2},\overline{2})=z_1z_2^{-1},\quad
w(1,1,\overline{2},\overline{1})=-z_1z_2^{-1}Q_1,\\
w(1,1,\overline{1},\overline{1})=1,\quad
w(1,2,\overline{2},\overline{2})=-z_1z_2^{-1}Q_1,\quad
w(1,2,\overline{2},\overline{1})=z_1z_2^{-1}Q_1^2,\quad
w(1,2,\overline{1},\overline{1})=-Q_1,\\
w(1,\overline{2},\overline{1},\overline{1})=Q_1Q_2,\quad
w(2,2,\overline{2},\overline{2})=1,\quad
w(2,2,\overline{2},\overline{1})=-Q_1,\quad
w(2,2,\overline{1},\overline{1})=z_1^{-1}z_2,\quad\\
w(2,\overline{2},\overline{1},\overline{1})=-z_1^{-1}z_2Q_2,\quad
w(\overline{2},\overline{2},\overline{1},\overline{1})=z_1^{-1}z_2^{-1}.
\end{gather*}
Therefore, we have
\begin{align*}
F_2&=\sum_{x_1\leq y_1<x_2\leq y_2}w(x_1,y_1,x_2,y_2)\\
&=
(1-Q_2)z_1z_2+(1-Q_1)^2z_1z_2^{-1}+2-2Q_1+Q_1Q_2+(1-Q_2)z_1^{-1}z_2+z_1^{-1}z_2^{-1}.    
\end{align*}
\end{example}

Since $w(k,\overline{k+1})=-w(k,\overline{k})$, the expression \eqref{eq:path_description_of_F} is slightly improved:
\begin{equation}\label{eq:imp_exp_of_F}
F_i=
\sum_{
\substack{
x_1\leq y_1<x_2\leq y_2<\dots<x_i\leq y_i\\
\text{if }(x_l,y_l)=(k,\overline{k+1})\text{ then } x_{l+1}=\overline{k}
}
}\prod_{l=1}^i
w(x_l,y_l).
\end{equation}
Note that we have by {\cite[Proposition~5.7]{KN}
}
\begin{equation}
F_i=F_{2n+1-i}\quad 
\text{for $1\le i\le 2n$.}
\end{equation}
The following theorem establishes the connection between the Lax matrix and the quantum $K$-theory  of type C, proved as \cite[Theorem~3.6]{KN}:
\begin{thm}\label{thm:relation_to_Ktheory}
For $i=1,\dots,2n$, the polynomials $F_i(z_1,\dots,z_n,Q_1,\dots,Q_n)$
coincide with the polynomials appearing in \cite{KN}
that define the torus-equivariant quantum $K$-ring of the type $\mathrm{C}_n$
flag variety.
More precisely, the elements
\[
F_i(z_1,\dots,z_n,Q_1,\dots,Q_n)
-
e_i(e^{\epsilon_1},\dots,e^{\epsilon_n},
e^{-\epsilon_1},\dots,e^{-\epsilon_n})
\]
agree with the generators of the defining ideal given in \cite{KN}.
\end{thm}

\subsection{The Lax equation}

Let
\[
\mathfrak{g}_+
=
\left\{
\left[
\begin{array}{c|c}
X & Y \\\hline
0 & Z
\end{array}\right]
\in \mathfrak{gl}_{2n}(\RR)
\;\middle|\;
\begin{array}{l}
X \text{ is upper triangular},\\
Z \text{ is an upper triangular matrix whose diagonal entries are $0$}
\end{array}
\hspace{-.5em}\right\}
\]
be the Lie algebra corresponding to $G_+$, and let
\[
\mathfrak{g}_-
=
\left\{
\left[
\begin{array}{c|c}
U & 0 \\\hline
V & W
\end{array}\right]
\in \mathfrak{gl}_{2n}(\RR)
\;\middle|\;
W = J U J
\right\}
\]
be the Lie algebra corresponding to $G_-$.
Then, the general linear Lie algebra $\mathfrak{gl}_{2n}(\RR)$ decomposes as $\mathfrak{gl}_{2n}(\RR)=\mathfrak{g}_+\oplus \mathfrak{g}_-$.
Let $\pi_\pm:\mathfrak{gl}_{2n}(\RR)\to \mathfrak{g}_\pm$ be the linear projections along the direct sum decomposition.

Let 
$\langle X,Y\rangle:=\mathrm{tr}(XY)$ for any $2n\times 2n$ matrices $X,Y$.
For any conjugation-invariant differentiable function $\varphi:GL_{2n}(\RR)\to \RR$ and $L\in GL_{2n}(\RR)$, we define the differential $d\varphi_L
$ as a unique element of $\mathfrak{gl}_{2n}(\RR)
$ satisfying
\[
\langle d\varphi_L,X\rangle
=\left.\frac{d}{d\epsilon}\varphi(e^{\epsilon X}L)\right|_{\epsilon=0},\qquad \forall X\in \mathfrak{gl}_{2n}(\RR).
\]
For example, if $\varphi(L)=\frac{1}{2}\mathrm{tr}(L^2)$, then we have $d\varphi_L=L$.
Via this identification, $L$ can also be regarded as an element of $\mathfrak{gl}_{2n}(\RR)$. 
Hence, the expression $\pi_{\pm}(L)$ makes sense.

Consider the \textit{Lax equation}:
\begin{equation}\label{eq:Lax_equation}
\frac{d}{dt}L=[L,\pi_+(L)].
\end{equation}
The phase space of \eqref{eq:Lax_equation} admits a Lie theoretic interpretation.
Let $\Gamma_1,\Gamma_2\subset GL_{2n}(\RR)$ be the affine subvarieties defined by
\begin{gather*}
\Gamma_1=\left\{
L\in GL_{2n}(\RR)\left|\;
\mbox{$L^{-1}$ is of the form } L^{-1}=
\left[
\begin{array}{cccc|cccccc}
\ast & \cdots & \ast & \ast & \ast & \cdots & \ast& \ast\\
\ast & \cdots & \ast & \ast & \ast & \cdots & \ast& \ast\\
0 & \ddots & \vdots & \vdots & \vdots & \vdots & \vdots& \vdots\\
0 & 0 & \ast & \ast & \ast & \cdots & \ast& \ast\\\hline
0 & \cdots & 0 & \ast & \ast& \cdots & \ast& \ast \\
0 & \cdots & 0 & 0 & 1& \cdots & \ast& \ast\\
\vdots & \vdots & \vdots & \vdots & 0 & \ddots& \vdots&\vdots\\
0 & \cdots & 0 & 0 & 0 & 0 &1& \ast\\
\end{array} 
\right]
\right\}\right.,\\
\Gamma_2=\left\{L=
\left.\left[
\begin{array}{c|c}
U & J\\\hline
\ast & W
\end{array}
\right]
\in GL_{2n}(\RR)\right|
\mbox{$L^{-1}$ is of the form }
L^{-1}=
\left[
\begin{array}{c|c}
JWJ & -J\\\hline
\ast & JUJ
\end{array}
\right]
\right\}.
\end{gather*}
Let
\[
\Gamma:=\Gamma_1\cap \Gamma_2
\]
be the intersection of these varieties.
One can verify that the Lax matrix $L$ is contained in $\Gamma$.

\begin{lemma}\label{lemma:Lax_decomposition_to_NBC}
For a generic element $L\in \Gamma$, there exists a string of $2n$ complex numbers 
$$(z_1,\dots,z_n,Q_1,\dots,Q_n)\in \RR^{2n}$$ such that $L=NBC^{-1}$ as in \eqref{eq:matrix_expression_of_L}.
\end{lemma}
\begin{proof}
By the Gauss decomposition, a generic element $L\in \Gamma$ can be decomposed as $L=Z^{-1}BY^{-1}$, where $Y$ (\textrm{resp}.~$Z$) is a block lower triangular (\textrm{resp}.~block upper triangular) matrix expressed as
\[
Y=\left[
\begin{array}{c|c}
Y_{11} &O\\\hline
Y_{21} &Y_{22}\\
\end{array}\right],\quad
Z=\left[
\begin{array}{c|c}
Z_{11} &Z_{12}\\\hline
O &Z_{22}\\
\end{array}
\right],\quad 
Y_{11}\in U_-,\ 
Y_{22}\in B_-,\ 
Z_{11}\in B_+,\ 
Z_{22}\in U_+. 
\]
Hence, the matrix $L^{-1}=YB^{-1}Z$ is expressed as
\[
L^{-1}=\left[
\begin{array}{c|c}
Y_{11}Z_{11} &Y_{11}(Z_{12}-JZ_{22})\\\hline
Y_{21}Z_{11} &Y_{21}(Z_{12}-JZ_{22})+Y_{22}Z_{22}
\end{array}
\right].
\]
Since $L\in \Gamma_1$, all entries of $Y_{21}$, except for the $(1,n)$-th one, are zero.
For the same reason, all entries of $Y_{11}$ and $Y_{22}$, except for the main-diagonal and sub-diagonal ones, are zero.
On the other hand, since $L\in \Gamma_2$, we have the equation $JY_{11}Z_{11}J=Z_{22}^{-1}Y_{22}^{-1}$, which implies that all entries of $Z^{-1}_{11}$ and $Z^{-1}_{22}$, except for the main-diagonal and sub-diagonal ones, are zero.
In particular, all sub-diagonal entries of $JY_{11}J$ and $Z_{22}^{-1}$ coincide. This fact leads to the equation $JY_{11}J=Z_{22}^{-1}$.
Because $Y_{11}(Z_{12}-JZ_{22})=-J$, we find $Z_{12}=O$.
Therefore, again from $L\in \Gamma_2$, we obtain $JY_{22}J=Z_{11}^{-1}$.

We eventually obtain
\[
Y=\left[
\begin{array}{c|c}
JZ_{22}^{-1}J &O\\\hline
Y_{21} &JZ_{11}^{-1}J\\
\end{array}\right],\quad
Z^{-1}=\left[
\begin{array}{c|c}
Z_{11}^{-1} & O\\\hline
O &Z_{22}^{-1}\\
\end{array}
\right],\quad Y_{21}=k\cdot E_{1,n}.
\]
for some $k\in \CC$.
By defining the parameters $z_1,\dots,z_n,Q_1,\dots,Q_n$ by
\begin{gather*}
z_i=(\text{the $(i,i)$-th entry of $Z_{11}^{-1}$}),\\
Q_i=z_i^{-1}\cdot (\text{the $(n-i,n-i+1)$-th entry of $Z_{22}^{-1}$}),\quad
Q_n=z_n^{-1}\cdot (\text{the $(1,n)$-th entry of $Y_{21}$}),
\end{gather*}
we obtain $N=Z^{-1}$ and $C=Y$, where $N$ and $C$ are the $n\times n$ matrices defined in \eqref{eq:matrix_expression_of_L}.
This leads to the desired expression.
\end{proof}

From Lemma \ref{lemma:Lax_decomposition_to_NBC}, one finds that $\Gamma$ contains an open dense subset isomorphic to $\RR^{2n}$.
This $2n$-dimensional subset can be regarded as the space of Lax matrices.

\subsection{The flow of motion}

The Lax equation \eqref{eq:Lax_equation} defines a flow on the $2n$-dimensional variety $\Gamma$.
To show this fact, we need to prove that the flow preserves $\Gamma$.

\begin{prop}\label{prop:adjoint_action}
The variety $\Gamma_1$ is preserved by the adjoint action $L\mapsto g_+Lg_+^{-1}$ with $g_+\in G_+$, and the variety $\Gamma_2$ is preserved by the adjoint action $L\mapsto g_-Lg_-^{-1}$ with $g_-\in G_-$.
\end{prop}
\begin{proof}
Direct calculations can verify these statements.
\end{proof}

By virtue of Propositions \ref{prop:mat_decomp} and \ref{prop:adjoint_action}, the general solution to the Lax equation \eqref{eq:Lax_equation} is constructed as follows.
Fix an initial state $L_0\in \Gamma$.
Let $\exp(L_0t)=a(t)^{-1}b(t)$ be the matrix decomposition with $a(t)\in G_+$ and $b(t)\in G_-$.

Define $L(t):=a(t)L_0a(t)^{-1}$.
Since $L_0\exp(L_0t)=\exp(L_0t)L_0$, we have
\begin{equation}\label{eq:L_t}
L(t)=a(t)L_0a(t)^{-1}=b(t)L_0b(t)^{-1}.
\end{equation}
By Proposition \ref{prop:adjoint_action}, we have $L(t)\in \Gamma=\Gamma_1\cap \Gamma_2$ for all $t$.

By differentiating both sides of \eqref{eq:L_t}, we obtain
\[
L'(t)=a'(t)L_0a(t)^{-1}-a(t)L_0a(t)^{-1}a'(t)a(t)^{-1}=[a'(t)a(t)^{-1},L(t)].
\]
On the other hand, by differentiating $\exp(L_0t)=a(t)^{-1}b(t)$, we obtain
\[
L_0\exp(L_0t)=-a(t)^{-1}a'(t)a(t)^{-1}b(t)+a(t)^{-1}b'(t),
\]
which implies
\[
L(t)
=-a'(t)a(t)^{-1}+b'(t)b(t)^{-1}.
\]
Since $-a'(t)a(t)^{-1}\in \mathfrak{g}_+$ and $b'(t)b(t)^{-1}\in \mathfrak{g}_-$, we obtain $\pi_+(L)=-a'(t)a(t)^{-1}$ and $\pi_-(L)=b'(t)b(t)^{-1}$.
Therefore, the matrix $L(t)\in \Gamma$ is a solution to the Lax equation \eqref{eq:Lax_equation}.
This means that the flow defined by the differential equation \eqref{eq:Lax_equation} preserves the phase space $\Gamma$.

\subsection{Computation of the Hamiltonian}

Direct calculations show that the $(i,i)$-entry of $\pi_+(L)$ is
\[
\begin{cases}
(1-Q_1)z_1-z_1^{-1}& (i=1),\\
(1-Q_i)z_i-(1-Q_{i-1})z_{i}^{-1}    & (1<i\leq n),\\
0 & (n<i\leq 2n).
\end{cases}
\]
By comparing the main-diagonal and the sub-diagonal entries on both sides of $(L^{-1})'=[L^{-1},\pi_+(L)]$, we obtain the system of differential equations:
\begin{equation}\label{eq:ordinal_diff_1}
\frac{Q_i'}{Q_i}=-(1-Q_{i-1})z_i^{-1}+(1-Q_i)(z_i+z_{i+1}^{-1})-(1-Q_{i+1})z_{i+1},\quad
\frac{Q_n'}{Q_n}=(1-Q_n)z_n-(1-Q_{n-1})z_n^{-1},
\end{equation}
\begin{equation}\label{eq:ordinal_diff_2}
\frac{z_i'}{z_i}=Q_i(z_i+z_{i+1}^{-1})-Q_{i-1}(z_{i-1}+z_{i}^{-1}),\quad
\frac{z_n'}{z_n}=Q_nz_n-Q_{n-1}(z_{n-1}+z_n^{-1}),
\end{equation}
where $1\leq i<n$ and $Q_0=z_0=0$.

Let $\{\cdot,\cdot\}$ be a Poisson bracket satisfying
\begin{equation}\label{eq:poisson_Q_z}
\{Q_i,Q_j\}=\{z_i,z_j\}=0,\quad \{Q_i,z_i\}=Q_iz_i,\quad \{Q_i,z_{i+1}\}=-Q_iz_{i+1},\quad \{Q_i,z_j\}=0\ \ (j\neq i,i+1).
\end{equation}
Then, the system \eqref{eq:ordinal_diff_1}, \eqref{eq:ordinal_diff_2} is rewritten as the Hamilton equation
\[
Q_i'=\{Q_i,H\},\quad
z_i'=\{z_i,H\},
\]
where $H$ is a Hamiltonian function
\[
H=F_1=\mathrm{tr}(L)=
(1-Q_1)z_1+(1-Q_2)z_2+\dots +(1-Q_n)z_n+(1-Q_{n-1})z_n^{-1}+
\dots+(1-Q_1)z_2^{-1}
+z_1^{-1}.
\]

Introduce the \textit{canonical variables} $(q_1,\dots,q_n,p_1,\dots,p_n)$ satisfying 
\[
\{q_i,q_j\}=
\{p_i,p_j\}=0,\qquad \{q_i,p_j\}=\delta_{i,j}
\]
and the variable change
\[
Q_i=
\begin{cases}
-e^{q_i-q_{i+1}} & (1\leq i<n),\\
-e^{q_n} & (i=n),
\end{cases}
\quad z_i=\exp\left( p_i+\frac{1}{2}\log
\left(
\frac{1+e^{q_{i-1}-q_i}}{1+e^{q_i-q_{i+1}}}
\right)
\right),\quad (q_0=-\infty,\,q_{n+1}=0).
\]
Direct calculations show that the change of variables is compatible with \eqref{eq:poisson_Q_z}.

Let $\bm{\alpha}_i=
\begin{cases}
\bm{e}_{i}-\bm{e}_{i+1} & (1\leq i<n)\\
\bm{e}_{n} & (i=n)
\end{cases}$ be the simple roots of type $\mathrm{B}_n$.
Set $\bm{q}:=(q_1,\dots,q_n)$.
Then, in the canonical variables, the Hamiltonian function $H$ is expressed as
\[
\begin{aligned}
H&=
2\sum_{i=1}^{n}\cosh(p_i)
\sqrt{1+\exp({\bm{\alpha}_{i-1}\cdot \bm{q})}}
\sqrt{1+\exp({\bm{\alpha}_{i}\cdot \bm{q})}}\\
&=
2\sum_{i=1}^{n-1}\cosh(p_i)
\sqrt{1+e^{q_{i-1}-q_{i}}}
\sqrt{1+e^{q_i-q_{i+1}}}
+2\cosh(p_n)
\sqrt{1+e^{q_{n-1}-q_{n}}}
\sqrt{1+e^{q_n}}.
\end{aligned}
\]
This Hamiltonian is naturally regarded as the relativistic Toda lattice
of type $\mathrm{B}_n$.
Indeed, it provides the natural $\mathrm{B}_n$-analogue of the Hamiltonians
of types $\mathrm{C}_n$ and $\mathrm{BC}_n$ given in
\cite[Eq.~(6.24), (6.25)]{ruijsenaars1990relativistic}.

\section{B\"{a}cklund transformation}\label{sec:Transformation}

From the matrix factorization of the Lax matrix, we naturally derive a \textit{B\"{a}cklund transformation} (also known as a \textit{Darboux-B\"{a}cklund transformation}), which is a change of variables that preserves the flow of motion.
The resulting transformation defines a birational map from $\Gamma$ to itself.

Let $C$ be the square matrix defined in \eqref{eq:matrix_expression_of_L}.
Let $a_i=Q_iz_i$, $b_i=z_i$, $M_i=1-\frac{a_i}{b_{i+1}}$, and $N_i=1-\frac{a_ia_{i+1}}{b_{i+1}b_{i+2}}$.
Following the procedure in the proof of Proposition \ref{prop:mat_decomp}, we decompose $C$ as $C=KR^{-1}$ with 
\begin{equation}\label{eq:KR}
K=\left[
\begin{array}{c|c}
K_{11}&O\\\hline
K_{12}&K_{22}
\end{array}\right]\in G_-,\qquad
R=\left[
\begin{array}{c|c}
R_{11}&O\\\hline
O&R_{22}
\end{array}
\right]\in G_+,
\end{equation}
where
\[
K_{11}=\begin{bmatrix}
\frac{b_1}{M_1} & 1&\\
\frac{a_1b_1}{M_1} & \frac{N_1b_2}{M_2}&1&\\
&\frac{M_1a_2b_2}{M_2} & \ddots&\ddots&\\
&&\ddots & \frac{N_{n-2}b_{n-1}}{M_{n-1}}&1\\
&&&\frac{M_{n-2}a_{n-1}b_{n-1}}{M_{n-1}} & b_n
\end{bmatrix},\quad K_{12}=M_{n-1}a_nb_n\cdot E_{1,n},\quad K_{22}=JK_{11}J
\]
and
\[
R_{11}=
\begin{bmatrix}
\frac{b_1}{M_1} & 1&\\
& \frac{M_1}{M_2}b_2&1&\\
&& \ddots&\ddots&\\
&&& \frac{M_{n-2}}{M_{n-1}}b_{n-1}&1\\
&&& & M_{n-1}b_n
\end{bmatrix},\quad
R_{22}=
\begin{bmatrix}
 1&\frac{M_{n-2}a_{n-1}b_{n-1}}{M_{n-1}b_n}\\
 &1&\ddots\\
 &&\ddots&\frac{M_1a_2b_2}{M_2b_3}\\
 &&& 1&\frac{a_1b_1}{M_1b_2}\\
 &&& &1
\end{bmatrix}.
\]

Then, the Lax matrix $L=NBC^{-1}$ can be factorized as $L=(NBR)K^{-1}$, where $NBR\in G_+$ and $K^{-1}\in G_-$.
By switching the positions and multiplying the matrices, we define the new matrix
\[
L^+:=K^{-1}(NBR).
\]
Since $L^+=(NBR)^{-1}L(NBR)=K^{-1}LK$, the matrix $L^+$ also belongs to the phase space $\Gamma$ by Proposition \ref{prop:adjoint_action}.
Then, by applying Lemma~\ref{lemma:Lax_decomposition_to_NBC} again, $L^+$ can be decomposed as
\[
L^+ = (N^+ B R^+)(K^+)^{-1},
\]
where the matrices $N^+$, $R^+$, and $K^+$ have a form similar to $N$, $R$, and $K$, respectively.
By comparing the matrix entries, we find that the matrices $N^+$, $R^+$, and $K^+$ are obtained from $N$, $R$, and $K$ by applying a birational transformation $(z_i,Q_i)\mapsto (z_i^+,Q_i^+)$:
\begin{equation}\label{eq:birational_map}
Q_i^+:=\frac{M_{i-1}M_{i+1}}{M_i^2}\frac{z_i^2}{z_{i+1}^2}Q_i,\qquad
z_i^+:=\frac{1-Q_{i-1}^+}{1-Q_{i}^+}\frac{M_{i-1}}{M_{i}}z_i,\qquad
(M_i=1-Q_i\frac{z_i}{z_{i+1}}),
\end{equation}
where $Q_0^+=0$ and $M_0=M_{n}=M_{n+1}=z_{n+1}=1$.

{
\begin{prop}
The transformation $L\mapsto L^+$ preserves the flow of motion \eqref{eq:Lax_equation}:
\[
\frac{d}{dt}L^+=[L^+,\pi_+(L^+)].
\]
\end{prop}
\begin{proof}
Let $L_0=L(0)$ be the initial value. 
Consider the matrix decomposition $L_0=M_0K_0^{-1}$ with $M_0\in G_+$, $K_0\in G_-$, and define $L_0^+:=K_0^{-1}L_0K_0$.
Let $\exp(L_0^+t)=(a^+)^{-1}b^+$ be the decomposition with $a^+\in G_+$, $b^+\in G_-$, and $\widehat{L}:=a^+L^+_0(a^+)^{-1}=b^+L^+_0(b^+)^{-1}$.
By construction, the matrix $\widehat{L}$ satisfies the differential equation $\frac{d}{dt}\widehat{L}=[\widehat{L},\pi_+(\widehat{L})]$ with initial value $L_0^+$.
Hence, it suffices to show 
\begin{equation}\label{eq:L+isLhat}
L^+=\widehat{L}.
\end{equation}

Let $\exp(L_0t)=a^{-1}b$, with $a\in G_+$ and $b\in G_-$.
Then, we have $\exp(L_0^+t)=\exp(K_0^{-1}L_0K_0t)=K_0^{-1}\exp(L_0t)K_0=K_0^{-1}a^{-1}bK_0$, which implies $K_0^{-1}a^{-1}bK_0=(a^+)^{-1}b^+$.
By using the equation $K_0^{-1}a^{-1}=(a^+)^{-1}b^+K_0^{-1}b^{-1}$, we have $L=aL_0a^{-1}=aM_0K_0^{-1}a^{-1}=aM_0(a^+)^{-1}\cdot b^+K_0^{-1}b^{-1}$.
Since $aM_0(a^+)^{-1}\in G_+$ and $b^+K_0^{-1}b^{-1}\in G_-$, we have $K^{-1}=b^+K_0^{-1}b^{-1}$.
Therefore, we conclude
\[
\begin{aligned}
L^+&=K^{-1}LK\\
&=b^+K_0^{-1}b^{-1}\cdot L\cdot bK_0(b^+)^{-1}\\
&=b^+K_0^{-1}b^{-1}\cdot bL_0b^{-1}\cdot bK_0(b^+)^{-1}\\
&=b^+L_0^+(b^+)^{-1}\\
&=\widehat{L},
\end{aligned}
\]
which implies \eqref{eq:L+isLhat}.
\end{proof}
}
Then, the map \eqref{eq:birational_map} defines a B\"{a}cklund transformation.
We can also regard the birational map \eqref{eq:birational_map} as a \textit{discrete-time relativistic Toda lattice of type $\mathrm{B}_n$}.

\begin{example}
When $n=2$, the B\"{a}cklund transformation $(z_i,Q_i)\mapsto (z_i^+,Q_i^+)$ is expressed as
\begin{equation}\label{eq:backlund_2}
\begin{gathered}
Q_1^+=\frac{z_1^2}{(z_2-Q_1z_1)^2}Q_1,\quad
Q_2^+=(z_2-Q_1z_1)z_2Q_2,\\
z_1^+=\frac{1}{1-Q_1^+}\frac{z_1z_2}{z_2-Q_1z_1},\quad
z_2^+=\frac{1-Q_1^+}{1-Q_2^+}(z_2-Q_1z_1).
\end{gathered}
\end{equation}
By direct calculations using \eqref{eq:ordinal_diff_1} and \eqref{eq:ordinal_diff_2}, we have
\begin{align*}
\frac{(Q_1^+)'}{Q_1^+}
&=2\frac{z_1'}{z_1}-2\frac{z_2'-(Q_1z_1)'}{z_2-Q_1z_1}+\frac{Q_1'}{Q_1}\\
&=
2\frac{Q_1z_1^2
+Q_1z_1z_2^{-1}}{z_2-Q_1z_1}
-z_1^{-1}+(1+Q_1)(z_1+z_2^{-1})-(1+Q_2)z_2\\\displaybreak[1]
&=
-(z_1^+)^{-1}+(1-Q_1^+)(z_1^++(z_{2}^+)^{-1})-(1-Q_{2}^+)z_{2}^+,\\
\frac{(Q_2^+)'}{Q_2^+}
&=
\frac{z_2'-(Q_1z_1)'}{z_2-Q_1z_1}+\frac{z_2'}{z_2}+\frac{Q_2'}{Q_2}\\
&=
-Q_1(z_1+z_2^{-1})+(1+Q_2)z_2-(1-Q_1)z_2^{-1}\\\displaybreak[1]
&=(1-Q_2^+)z_2^+-(1-Q_1^+)(z_2^+)^{-1},\\
\frac{(z_1^+)'}{z_1^+}
&=
-\frac{(1-Q_1^+)'}{1-Q_1^+}-\frac{z_2'-(Q_1z_1)'}{z_2-Q_1z_1}+\frac{z_1'}{z_1}+\frac{z_2'}{z_2}\\
&=\frac{(Q_1^+)'}{1-Q_1^+}+\frac{Q_1z_1^2+Q_1z_1z_2^{-1}}{z_2-Q_1z_1}\\\displaybreak[1]
&=Q_1^+(z_1^++(z_2^+)^{-1}),\\
\frac{(z_2^+)'}{z_2^+}
&=
\frac{(1-Q_1^+)'}{1-Q_1^+}-
\frac{(1-Q_2^+)'}{1-Q_2^+}+\frac{z_2'-(Q_1z_1)'}{z_2-Q_1z_1}\\
&=
-\frac{(Q_1^+)'}{1-Q_1^+}+
\frac{(Q_2^+)'}{1-Q_2^+}+Q_2z_2-\frac{Q_1z_1^2+Q_1z_1z_2^{-1}}{z_2-Q_1z_1}\\
&=Q_2^+z_2^+-Q_1^+(z_1^++(z_2^+)^{-1}).
\end{align*}
Hence, the B\"{a}cklund transformation \eqref{eq:backlund_2} preserves the flow of the relativistic Toda lattice.
\end{example}

\begin{example}
It is straightforward to check that the transformation \eqref{eq:backlund_2} preserves the Hamiltonian function $H=(1-Q_1)z_1+(1-Q_2)z_2+(1-Q_1)z_2^{-1}+z_1^{-1}$.
\end{example}

\section{Summary and discussions}

In this paper, we have presented a new $2n\times 2n$ Lax pair for the relativistic Toda lattice of type $\mathrm{B}_n$.
The coefficients of its characteristic polynomial admit a combinatorial expression in terms of a weighted graph.
Via this expression, we can compare them with the defining ideal of the quantum $K$-ring of the type C flag variety.
The Lax equation admits a matrix-based analysis in the same style as that by Kostant for the (non-relativistic) Toda lattice.

It is expected that our Lax matrix would provide an explicit connection between the quantum $K$-theory of type C and integrable systems.
For example, the B\"{a}cklund transformation defines an extremely nontrivial birational transformation of the equivariant quantum $K$-theory.
In future work, we plan to investigate the algebraic structure of the torus-equivariant quantum $K$-theory of type C, such as a $K$-Peterson isomorphism, Schubert calculus, and $K$-theoretic symmetric functions.

\appendix

\section{Proof of Equation \eqref{eq:path_description_of_F}}\label{sec:appendix}
In this appendix, we give a proof of \eqref{eq:path_description_of_F}.
Let $\mathbf{Q}_i=Q_1Q_2\cdots Q_i$ and $\mathbf{z}_i=z_1z_2\cdots z_i$.
Define
\[
\mathbf{a}_i:=\mathbf{Q}_{i-1}\cdot \mathbf{z}_{i-1}, \quad
\mathbf{b}_i=
\mathbf{Q}_{n}\cdot \mathbf{z}_{n-i}
\]
and $D=\mathrm{diag}(\mathbf{a}_1,\dots,\mathbf{a}_n,\mathbf{b}_1,\dots,\mathbf{b}_n)$, $Z=\mathrm{diag}(1,\dots,1,z_n,z_{n-1},\dots,z_1)$.
Then, the matrix $C$ defined in \eqref{eq:matrix_expression_of_L} is factorized as 
\[
C=Z\cdot (D\Lambda D^{-1}),\quad \text{where }\quad \Lambda=\sum_{k=1}^{2n}E_{kk}+ \sum_{k=1}^{2n-1}E_{k+1,k}.
\]
Therefore, the characteristic polynomial $\det(\lambda E-L)$ is rewritten as
\[
\det(\lambda E-L)
=\det(\lambda C-NB)\det(C^{-1})
=\det(Z)\cdot \det(\lambda D\Lambda D^{-1}-Z^{-1}NB)\cdot\det(C^{-1}).
\]
Since $\det Z=\det C$, we obtain
\[
\det(\lambda E-L)
=
\det(\lambda \Lambda-D^{-1}Z^{-1}NBD).
\]

Let $M:=D^{-1}Z^{-1}NBD$.
Direct calculations show that $M$ is factorized as
\[
\begin{bmatrix}
z_1\\
& \ddots\\
& &z_n\\
& & &z_n^{-1}\\
& & &&\ddots\\
& & &&&z_1^{-1}\\
\end{bmatrix}
\begin{bmatrix}
1&Q_1\\
 & 1 & \ddots \\
 & & \ddots & Q_{n-1}\\
 & & & 1 & 0\\ 
& & & & 1 & Q_{n-1}\\ 
&& & & & 1 & \ddots\\ 
&& & & & & \ddots&Q_1\\ 
&& & & & & &1
\end{bmatrix}
\left[
\begin{array}{cccccc}
1&& &  & & \mathbf{Q}_1\\
&\ddots& & & \rotatebox{-45}{$\vdots$}\\
&&1 & \mathbf{Q}_n\\
& & & 1\\
& & & & \ddots\\
& & & & & 1\\
\end{array}
\right].
\]

Consider the weighted graph discussed in Section \ref{sec:Int_system}. (An example for $n=3$ is displayed in Figure \ref{fig:graph_description}.)
In this appendix, we identify $I=\{1,2,\dots,n,\overline{n},\dots,\overline{2},\overline{1}\}$ with $\{1,2,\dots,2n\}$ by identifying $\overline{i}$ with $2n+1-i$.
Then, by the LGV-theorem, we find that the $(p,q)$-entry of $M$ is equal to $(-1)^{q-p}w(q,p)$.
On the other hand, the polynomial $\det(\lambda \Lambda -M)$ is expanded as
\begin{align}
\det(\lambda \Lambda -M)
&=\sum_{\sigma\in S_{2n}}\mathrm{sgn}(\sigma)
\prod_{k=1}^{2n} \{\lambda\cdot (\delta_{k,\sigma(k)}+\delta_{k,\sigma(k)+1})-(-1)^{\sigma(k)-k}w(k,\sigma(k))\}    \\
&=
\sum_{i=0}^{2n}(-1)^{2n-i}\lambda^{2n-i}
\sum_{
\substack{
K\subset \{1,\dots,2n\},\\
\# K=i
}
}
\sum_{\tau}\mathrm{sgn}(\tau)
\prod_{k\in K} (-1)^{\tau(k)-k}w(k,\tau(k)),\label{eq:expand_char}
\end{align}
where $\tau$ runs over the set 
\[
\mathfrak{X}_K:=
\left\{\tau\in S_{2n}\left|
\begin{aligned}
\text{ (i) }&\tau(l)-l\in \{0,-1\}\text{ if } l\notin K\\
\text{ (ii) }&\tau(k)\geq k\text{ if } k\in K.
\end{aligned}
\right\}\right..
\]
Note that if $\tau\in \mathfrak{X}_K$, then we have 
\begin{equation}\label{eq:rel_of_XK}
k_1<k_2 \ \Rightarrow\ k_1\leq \tau(k_1)<k_2\leq \tau(k_2)\quad 
\end{equation}
for $k_1,k_2\in K$.
Therefore, the collection of paths
\[
\left.\{\gamma(k,\tau(k))\right|k\in K\}
\]
is automatically non-intersecting.
This implies that, if $\tau\in \mathfrak{X}_K$, then the signature $\mathrm{sign}(\tau)$ is equal to $\prod_{k\in K} (-1)^{\tau(k)-k}$.
Hence, we have
\begin{align*}
\sum_{\tau}\mathrm{sgn}(\tau)
\prod_{k\in K} (-1)^{\tau(k)-k}w(k,\tau(k))
&=
\sum_{\tau}
\prod_{k\in K} w(k,\tau(k)).
\end{align*}
Substituting this to \eqref{eq:expand_char}, we obtain
\[
F_i=
\sum_{\tau}
\prod_{k\in K} w(k,\tau(k)).
\]
From \eqref{eq:rel_of_XK}, we deduce that
\[
F_i=\sum_{x_1\leq y_1<\dots<x_i\leq y_i} \prod_k w(x_k,y_k),
\]
which concludes \eqref{eq:path_description_of_F}.

\end{document}